\documentclass[letterpaper, 10 pt, conference]{ieeeconf}  % Comment this line out if you need a4paper

\IEEEoverridecommandlockouts                              % This command is only needed if 
\newtheorem{theorem}{Theorem}
\newtheorem{lemma}{Lemma}
\newtheorem{remark}{Remark}
\newtheorem{definition}{Definition}

\newtheorem{problem}{Problem}
\newtheorem{procedure}{Procedure}

\newtheorem{example}{Example}

\usepackage{comment}

\usepackage{comment}
\usepackage{subfigure}
\newcommand{\abs}[1]{\ensuremath{\left\vert#1\right\vert}}

\newcommand{\cfunof}[1]{\ensuremath{\left\{#1\right\}}}

\newcommand{\C}{\ensuremath{\mathbb{C}}}

\newcommand{\funof}[1]{\ensuremath{\left(#1\right)}}

\newcommand{\Hfty}{\ensuremath{\mathscr{H}_{\infty}}}

\newcommand{\jw}{\ensuremath{\left(j\omega\right)}}

\newcommand{\Ltwo}[1][]{\ensuremath{\mathscr{L}_{2}^{#1}}}

\newcommand{\norm}[1]{\ensuremath{\left\Vert #1 \right\Vert}}

\newcommand{\Rat}{\ensuremath{\mathscr{R}}}

\newcommand{\R}{\ensuremath{\mathbb{R}}}

\newcommand{\s}{\ensuremath{\left(s\right)}}
\newcommand{\sqfunof}[1]{\ensuremath{\left[#1\right]}}

\newcommand{\tm}{\ensuremath{\left(t\right)}}

\usepackage[cmex10]{amsmath}
\usepackage{amsfonts}
\usepackage{amssymb}
\usepackage{mathtools}
\usepackage{nicefrac} 
\usepackage{mathrsfs}
\usepackage{cleveref,graphicx}
\usepackage[toc,acronym,nonumberlist]{glossaries}
\usepackage{enumerate,color}

\usepackage{tikz}

\usepackage{pgfplots} 

\pgfplotsset{compat=newest,width=\columnwidth} 
\pgfplotsset{plot coordinates/math parser=false} 
\newlength\fheight 
\newlength\fwidth 

\usepackage{tikz}
\usetikzlibrary{arrows,shapes,positioning}

\usepackage{cleveref}
\usepackage{autonum}
\usepackage{glossaries}
\crefname{problem}{problem}{problems}
\crefname{proposition}{proposition}{propositions}
\crefname{procedure}{procedure}{procedures}
\crefname{assumption}{assumption}{assumptions}
\newacronym{lti}{LTI}{linear time invariant}
\tikzset{every picture/.style=thick}

\title{\LARGE \bf
On the Optimal Control of Relaxation Systems
}

\author{Richard Pates, Carolina Bergeling and Anders Rantzer% <-this % stops a space
\thanks{The authors are with the Department of Automatic Control, Lund University, Box 118, {SE-221 00 Lund}, Sweden. They are members of the LCCC Linnaeus Center and the ELLIIT Excellence Center at Lund University.} 
\thanks{This work was supported by the Swedish Foundation for
Strategic Research, and the Swedish Research Council through the LCCC Linnaeus Center.}% <-this % stops a space
}

\makeatletter
\g@addto@macro\normalsize{%
 \setlength\abovedisplayskip{5pt plus 5pt minus 0pt} %2pt plus
 \setlength\belowdisplayskip{5pt plus 5pt minus 0pt} %2pt plus
 \setlength\abovedisplayshortskip{5pt plus 5pt minus 0pt} %2pt plus
 \setlength\belowdisplayshortskip{5pt plus 5pt minus 0pt} %2pt plus
}
\makeatother

\begin{document}

\maketitle
\thispagestyle{empty}
\pagestyle{empty}

\begin{abstract}
 The relaxation systems are an important subclass of the passive systems that arise naturally in applications. We exploit the fact that they have highly structured state-space realisations to derive analytical solutions to some simple H-infinity type optimal control problems. The resulting controllers are also relaxation systems, and often sparse. This makes them ideal candidates for applications in large-scale problems, which we demonstrate by designing simple, sparse, electrical circuits to optimally control large inductive networks and to solve linear regression problems.
\end{abstract}

%%%%%%%%%%%%%%%%%%%%%%%%%%%%%%%%%%%%%%%%%%%%%%%%%%%%%%%%%%%%%%%%%%%%%%%%%%%%%%%%
\section{Introduction}

%what studied
In this paper we consider the problem of designing optimal controllers for relaxation systems. Such systems play an important role in applications, and correspond to \cite{Wil72}:
\begin{enumerate}
    \item Reciprocal electrical networks with only one type of energy storage element (i.e. only inductors or only capacitors).
    \item Mechanical systems in which inertial effects may be neglected.
    \item Viscoelastic systems.
    \item Thermal systems.
\end{enumerate}
This makes them ideal candidates for modelling a range of simple networks  and optimisation algorithms, including the single commodity flow problem, symmetric consensus algorithms and heating networks \cite{SMR67,RBA05,BK06,BXM07}.

%what known
In the 1970s Jan Willems made several fundamental contributions on the realisability and synthesis of relaxation systems \cite{Wil72,Wil76,MW75}. In particular he demonstrated that they have highly structured state-space realisations. He used this property to connect several important reciprocity theorems from physics to the theory of dissipative systems, as well as to solve some problems in electrical network synthesis.

Our main contribution is to show that the same inherent structure in the realisations of relaxations systems can be exploited to solve two optimal control problems analytically. In particular we build on the techniques in \cite{BPR19} to show that if the system with dynamics
\[
\hat{y}\s=G\s\hat{u}\s
\] 
is of the relaxation type, then the control law
\begin{equation}\label{eq:c1}
\hat{u}\s=-\alpha^{-1}G\funof{0}\hat{y}\s
\end{equation}
minimises
\[
\int_{0}^\infty{}y\tm{}^Ty\tm+\alpha^2{}u\tm^Tu\tm{}\,dt
\]
over a set of bounded \Ltwo{}-norm disturbances. We also show that a similar energy-based performance measure is minimised by
\begin{equation}\label{eq:c2}
\hat{u}\s=-\alpha^{-1}\hat{y}\s.
\end{equation}
These results are presented in \Cref{sec:res}.

The simple analytical nature of these controllers makes them ideal candidates for applications to large-scale problems. This is because the control laws \cref{eq:c1,eq:c2} are simple to update if the network changes, and globally optimal. Furthermore they are at least as sparse as $G\funof{0}$ and can be synthesised with resistive circuits that inherit the underlying structure of the original system. This will be illustrated in \Cref{sec:ex}, where we will show how to design simple electrical circuits to optimally control large inductive networks and to solve least squares problems.

\section*{Notation}

$\Rat^{n\times{}m}$ denotes an $n\times{}m$ matrix of proper real rational transfer functions, and $\hat{y}\s$ the one-sided Laplace transform of a signal $y\tm:[0,\infty)\rightarrow{}\R^n$. A transfer function $G\in\Rat^{n\times{}m}$ has a realisation
\[
\Sigma_G=\sqfunof{\begin{array}{c|c}
        A & B \\\hline
        C & D
    \end{array}}
\]
if
\[
G\s=C\funof{sI-A}^{-1}B+D.
\]
A realisation is said to be minimal if $\funof{A,B}$ is controllable and \funof{A,C} is observable. Finally, $M^\dagger$ denotes the Moore-Penrose pseudo-inverse of a matrix with complex entries, and $M\succeq{}0$ and $M\succ{}0$ denote that such a matrix is both Hermitian, and positive semi-definite or positive definite, respectively.

\section{Preliminaries}

The relaxation systems, so called because of their close connections with the relaxation function from physics, are the input-output LTI systems with completely monotone impulse responses. Dynamcially they correponsd to the systems that exhibit no oscillatory behaviour. Jan Willems made several fundamental contributions on their realisability in the 70s, see \cite{Wil72,Wil76,MW75}. We will summarise and illustrate the properties of relaxation systems that we require in this section.

A matrix valued function
\[
W\funof{\cdot}:[0,\infty)\rightarrow{}\R^{m\times{}m}
\]
is said to be completely monotone \cite{Wid46} if for all $t>0$ and $n=0,1,2,\ldots{}$,
\[
\funof{-1}^n\frac{d^nW\tm}{dt^n}\succeq{}0.
\]
Basic examples include
\[
e^{-t}\,\text{and}\,\ln\funof{1+1/t},
\]
and if $A\preceq{}0$, $e^{At}$. We now formally define the relaxation systems.

\begin{definition}
Let $G\in\Rat^{m\times{}m}$ be the transfer function of a continuous time system with impulse response $D\delta\tm+W\tm$, where $\delta\tm$ is the Dirac delta function. $G$ is said to be a relaxation system if $D\succeq{}0$ and $W\tm$ is a completely monotone function.
\end{definition}

One of Willems' central contributions was to demonstrate that such systems have highly structured state-space realisations, and special storage functions that can be physically motivated. The relevant result for our purposes is the following, which is essentially just a restatement of \cite[Theorem 9]{Wil72}.

\begin{theorem}\label{thm:1}
Let $G\in\Rat^{m\times{}m}$ be the transfer function of a continuous time system. The following are equivalent:
\begin{enumerate}[(i)]
    \item $G$ is a relaxation system.
    \item There exist matrices $A,B$ and $D$ such that $\funof{A,B}$ is controllable, $A\preceq{}0$, $D\succeq{}0$ and
    \[
    \Sigma_G=\sqfunof{\begin{array}{c|c}
        A & B \\\hline
        B^T & D
    \end{array}}.
    \]
    \item Given any minimal realisation
    \[
    \Sigma_G=\sqfunof{\begin{array}{c|c}
        A & B \\\hline
        C & D
    \end{array}},
    \]
    the matrix $D$ is positive semi-definite, and there exists a $Q\succ{}0$ such that $QA=A^TQ\preceq{}0$ and $QB=C^T.$
\end{enumerate}
\end{theorem}

For the special symmetric realisation in (ii), the $Q$ from part (iii) equals $I$. However even in the general case it is always unique and can be calculated \cite[Lemma 3]{Wil72} according to
\[
Q=\begin{bmatrix}
    C^T&A^TC^T&\ldots{}\funof{A^T}^{n-1}
    \end{bmatrix}\begin{bmatrix}
    B&AB&\ldots{}A^{n-1}B
    \end{bmatrix}^\dagger{}.
\]
The matrix $Q$ has many appealing interpretations in the context of dissipativity theory. A detailed discussion of this would take us too far, however for the purposes of this paper it suffices to say that the quantity 
\[
V\funof{x}=\frac{1}{2}x^TQx
\]
corresponds directly to the energy stored internally in the system, and although $Q$ depends on the particular realisation of $G$, $V\funof{x}$ is specified entirely by the input-output behaviour of $G$. By factoring $Q=S^TS$ it can also be used to map an arbitrary minimal realisation into the symmetric form via
\begin{equation}\label{eq:simtr}
    \sqfunof{\begin{array}{c|c}
        A & B \\\hline
        C & D
    \end{array}}\mapsto{}\sqfunof{\begin{array}{c|c}
        SAS^{-1} & SB \\\hline
        CS^{-1} & D
    \end{array}}.
\end{equation}
We will highlight these features in the simple example below, and encourage the interested reader to consult \cite[\S{}10-12]{Wil72}.

\begin{figure}
    \centering
    \begin{tikzpicture}[>=stealth]
    \draw[-, thick] (1,0) -- (3,0);
    \draw[-, thick] (1,2) -- (3,2);
    
    \draw[-, thick] (2,2) -- (2,1.5);
    \draw[-, thick] (2,0) -- (2,0.5);
    \draw[->, thick] (2,0.5) -- node[right, yshift=0mm] {$i_R$} (2,0.15);
    \draw[decoration=zigzag,decorate, thick] (2,.5) -- (2,1.5);
    \node (R) at (1.5,1) {$R$};
    
    \draw[-, thick] (3,0) -- (3,2);
    \draw[white,fill=white] (2.8,0.9) rectangle (3.2,1.1);
    \draw[-, thick] (2.7,0.9) -- (3.3,0.9);
    \draw[-, thick] (2.7,1.1) -- (3.3,1.1);
    \node (C) at (3.75,1) {$C$};
    
    \draw[->, thick] (1,0.15) -- node[left] {$v$} (1,1.85);
    \draw[<-, thick] (0.25,0) -- node[above] {$i$} (0.85,0);
    \draw[->, thick] (0.25,2) -- node[above] {$i$} (0.85,2);
    
    \draw[black, fill=white] (1,0) circle (0.05);
    \draw[black, fill=white] (1,2) circle (0.05);
    
    \end{tikzpicture}
    \caption{RC-circuit studied in \Cref{ex:1}}
    \label{fig:1}
\end{figure}
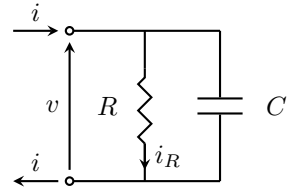

\begin{example}\label{ex:1}
Consider the simple RC-circuit shown in \Cref{fig:1}. This system is governed by the equations
\begin{equation}\label{eq:capeq}
\begin{aligned}
    \dot{q}&=i-i_R,\\
    v&=\frac{1}{C}q=Ri_R,
    \end{aligned}
\end{equation}
where $i$ and $v$ are the current through and voltage across the terminals, $q$ the charge on the capacitor with capacitance $C$, and $i_R$ the current through the resistor with resistance $R$. The transfer function $G:i\rightarrow{}v$ for this system equals
\begin{equation}\label{eq:tfcaprel}
G\s=\frac{1}{Cs+1/R},
\end{equation}
and has realisation
\begin{equation}\label{eq:caprel}
\Sigma_G=\sqfunof{\begin{array}{c|c}
        -1/{RC} & 1 \\\hline
        1/C & 0
    \end{array}}.
\end{equation}
Setting $Q=1/C$ and applying \Cref{thm:1} shows that $G$ is a relaxation system, which is to be expected since it contains only one type of storage element. Furthermore
\[
V\funof{x}\equiv{}\frac{1}{2}\frac{q^2}{C},
\]
which is the familiar equation for the energy stored in a capacitor, and the similarity transform \cref{eq:simtr} given by $S=1/\sqrt{C}$ clearly takes \cref{eq:caprel} into its symmetric form.
\end{example}

\section{Results}\label{sec:res}

In this section we solve two simple optimal control problems for the following LTI system
\begin{equation}\label{eq:p1}
\begin{aligned}
\dot{x}&=Ax+Bu+w,\;x\funof{0}=0,\\
y&=Cx+Du,
\end{aligned}
\end{equation}
that apply when it realises a relaxation system. In the above $x\in\R^n,y\in\R^p,u\in\R^m$ and $w\in\R^n$ denote the state, output, input and disturbance respectively, and $A,B,C,D$ are matrices of compatible dimension with real entries. We consider the problem of designing an internally stabilising control law
\begin{equation}\label{eq:p2}
\hat{u}\s=-K\s{}\hat{y}\s
\end{equation}
to minimise two different performance measures in the face of the disturbance $w$. We assume that the disturbance is from the following class
\[
\mathcal{W}_Q\coloneqq{}\cfunof{w\tm:\int_{0}^\infty{}w\tm^TQw\tm{}\,dt\leq{}1},
\]
where $Q$ is a positive definite matrix.

First we consider the following optimal control problem: 
\begin{problem}\label{prob:1}
Let $\alpha>0$. Minimise
\begin{equation}\label{eq:op1}
\sup_{w\in\mathcal{W}_Q}\int_{0}^\infty{}y\tm{}^Ty\tm+\alpha^2{}u\tm^Tu\tm{}\,dt
\end{equation}
subject to \cref{eq:p1,eq:p2} over stabilising $K\in\Rat^{m\times{}m}$.
\end{problem}
In words, the objective is to design the controller to regulate the output $y$ in the presence of disturbances $w\in\mathcal{W}_Q$. The second term in \cref{eq:op1} penalises the amount of control effort required to achieve this, and the size of $\alpha$ can be chosen to balance these competing objectives.

The second problem we consider is the following:
\begin{problem}\label{prob:2}
Let $\alpha>0$. Minimise
\begin{equation}\label{eq:op2}
\sup_{w\in\mathcal{W}_Q}\int_{0}^\infty{}y\tm{}^Ty\tm+\alpha^2{}u\tm^T\bar{y}\tm{}\,dt,
\end{equation}
where
\[
\bar{y}\tm=y\tm-\int_0^tCe^{A\funof{t-\tau}}w\funof{\tau{}}\,dt,
\]
subject to \cref{eq:p1,eq:p2} over stabilising $K\in\Rat^{m\times{}m}$.
\end{problem}
The objective is very similar to \Cref{prob:1}. The only difference is that the term penalising the control effort has been replaced with a penalty on $u\tm^T{}\bar{y}\tm$. Note that $\bar{y}$ is nothing but the part of the output that is caused by the input $u$. The motivation for this is that if \cref{eq:p1} realises a relaxation system, then typically the quantity $u\tm{}^T\bar{y}\tm{}$ is the product of a current and voltage (or their analogues), and has the units of power. Therefore
\[
\int_0^\infty{}u\tm^T\bar{y}\tm{}\,dt
\]
corresponds to the energy supplied to the system by the controller, which is arguably a more natural way to penalise the control effort.

\begin{remark}
\Cref{prob:1} is equivalent to a standard \Hfty{} optimal control problem, but \Cref{prob:2} is not.
\end{remark}

The following theorem is the main result of this paper, and shows that if
\[
\sqfunof{\begin{array}{c|c}
        A & B \\\hline
        C & D
    \end{array}}=\Sigma_G
\]
is the realisation of a relaxation system $G$ with storage
\[
V\funof{x}=\frac{1}{2}x^TQx,
\]
then both \Cref{prob:1,prob:2} can be solved analytically (the constraints on the realisation in the theorem statement exactly match those in \Cref{thm:1}(iii)). Note in particular that the optimal controllers are themselves relaxation systems, are independent of the realisation of $G$, are at least as sparse as $G\funof{0}$ and can be synthesised with networks of resistors (and transformers) that inherit the underlying structure of the original system. These features will be illustrated in \Cref{sec:ex}.

\begin{theorem}\label{thm:main}
If $AQ=QA^T\preceq{}0$, $QB=C^T$ and $D\succeq{}0$, then:
\begin{enumerate}
    \item $K\equiv{}\alpha^{-1}\funof{D-CA^{-1}B}$ solves \Cref{prob:1}.
    \item $K\equiv{}\alpha^{-1}I$ solves \Cref{prob:2}.
\end{enumerate}
\end{theorem}

Before proving the result we will illustrate its meaning on the system from \Cref{ex:1}.

\begin{example}
Applying \Cref{thm:main} to the system in \Cref{ex:1} with realisation \cref{eq:caprel} and $Q\equiv{}1/C$ shows that the controller
\[
K\equiv{}R/\alpha
\]
solves \Cref{prob:1}, and the controller
\[
K\equiv{}1/\alpha
\]
solves \Cref{prob:2}. Both controllers are independent of the realisation, provided the matrix $Q$ is updated accordingly. Observe that choosing $Q$ to satisfy the conditions of the theorem normalises the size of the disturbance to match the physical properties of the energy storage elements.
\end{example}

\begin{proof}
The proof will be in two stages. We will first show that the given controllers are optimal with respect to a restricted class of disturbances. We will then exploit the properties of relaxation systems to show that the same level of performance is achieved even when disturbances in the full class are allowed (that is, the worst case disturbances in $\mathcal{W}_Q$ are contained in the restricted class). Stability will be tacitly assumed throughout, and is guaranteed by the passivity theorem (both the plant and controller are relaxation systems, which are automatically passive).

\textit{Stage 1:}
Let 
\[
H_T\tm=\begin{cases}
1/T&\text{if $0\leq{}t\leq{}T$,}\\
0&\text{otherwise},
\end{cases}
\]
and define the class of disturbances
\begin{equation}\label{eq:disres}
\mathcal{W}_Q^T=\cfunof{w\tm:w\tm=H_T\tm{}v,v^TQv=1}.
\end{equation}
We will now show that the controllers in 1) and 2) minimise the performance criteria in \Cref{prob:1,prob:2} over all disturbances $w\in\mathcal{W}_Q^T$ in the limit $T\rightarrow{}\infty$. The key feature in this argument is that this restriction reduces the synthesis problem into a matrix minimisation problem that can be solved using least squares techniques. Since for any $T>0$, $\mathcal{W}_Q^T\subset\mathcal{W}_Q$, the minimum value of the cost over this class of disturbances can be no larger than the cost in \Cref{prob:1,prob:2}.

We now proceed to solve \Cref{prob:1,prob:2} under the restriction that $w\in\mathcal{W}_Q^T$. The system dynamics impose the following constraint between $\hat{u},\hat{y}$ and $\hat{w}$:
\begin{equation}\label{eq:r4}
\begin{bmatrix}
\hat{y}\s\\\hat{u}\s
\end{bmatrix}\!\!=\!\!\begin{bmatrix}
I\\-K\s
\end{bmatrix}\!\funof{I+G\s{}K\s}^{-1}\!C\funof{sI-A}^{-1}\!\!\hat{w}\s,
\end{equation}
where
\[
G\s=C\funof{sI-A}^{-1}B+D.
\]
A standard argument (e.g. \cite[Chapter 1]{Vin00}) shows that if
\begin{equation}\label{eq:basedis}
w\tm=vH_T\tm{},
\end{equation}
where $v\in\R^m$, then
\begin{equation}\label{eq:cri1}
\lim_{T\rightarrow{}\infty}\int_{0}^\infty{}y\tm{}^Ty\tm+\alpha^2{}u\tm^Tu\tm{}\,dt=z^T\begin{bmatrix}
I&0\\0&\alpha^2
\end{bmatrix}z,
\end{equation}
where
\begin{equation}
z=\begin{bmatrix}
I\\-K\s
\end{bmatrix}\funof{I+G\s{}K\s}C\funof{sI-A}^{-1}\Bigg\vert_{s=0} v.
\end{equation}
Note that this is saying nothing more than the size of the response of a stable system to a step input is given by the DC gain of the system. It then follows from \Cref{lem:lb}, which is stated and proved in the Appendix, that $K\equiv{}\alpha^{-1}G\funof{0}^T$ minimises
\[
\lim_{T\rightarrow{}\infty}\sup_{w\in\mathcal{W}_Q^T}\int_{0}^\infty{}y\tm{}^Ty\tm+\alpha^2{}u\tm^Tu\tm{}\,dt
\]
subject to \cref{eq:p1,eq:p2}. By \Cref{thm:1}(ii), $G\funof{0}=G\funof{0}^T$, which proves that the controller in 1) is optimal for \Cref{prob:1} when the disturbances are restricted to lie in $\mathcal{W}_Q^\infty{}$. 

A similar argument can be used on \Cref{prob:2}. To see this observe that $\hat{\bar{y}}\s=G\s\hat{u}\s$. Therefore just as before, if the disturbance is given by \cref{eq:basedis}, in the limit $T\rightarrow{}\infty$
\begin{equation}\label{eq:cri2}
\int_{0}^\infty{}y\tm{}^Ty\tm+\alpha^2{}u\tm^T\bar{y}\tm{}\,dt=z^T\begin{bmatrix}
I&0\\0&\alpha^2G\funof{0}
\end{bmatrix}z.
\end{equation}
Since by \Cref{thm:1}(ii), $G\funof{0}\succeq{}0$ it also follows from \Cref{lem:lb} that $K\equiv{}\alpha^{-1}I$ solves
\Cref{prob:2} over disturbances in $\mathcal{W}_Q^\infty{}$ (technically this requires that $G\funof{0}\succ{}0$, but a simple limit argument can be used to cover the semi-definite case).

\textit{Stage 2:} We will now show that whenever $K\s=\bar{K}\succeq{}0$ is stabilising, the disturbances of the form in \cref{eq:disres} are the worst-case. This will prove that the controllers in 1) and 2) are optimal since they are both stabilising, positive semi-definite, and optimal over disturbances in \cref{eq:disres}.

We will first consider \Cref{prob:1}. Note that given any controller the performance criterion in this problem equals
\[
\norm{\begin{bmatrix}
I\\-\alpha{}^{-1}K\s
\end{bmatrix}\funof{I+G\s{}K\s}^{-1}C\funof{sI-A}^{-1}\sqrt{Q}^{-1}}_\infty,
\]
where $\sqrt{\cdot{}}$ denotes the positive definite matrix square root. Standard algebraic manipulations show that the transfer function in the above equals
\begin{equation}\label{eq:11}
\begin{bmatrix}
I\\-\alpha{}^{-1}K\s
\end{bmatrix}\funof{I+DK\s}^{-1}C\sqrt{Q}^{-1}M\s,
\end{equation}
where
\[
M\s\!=\!\!\sqrt{Q}\funof{sI-A+BK\s\funof{I+DK\s}^{-1}\!\!C}^{-1}\!\!\!\!\!\!\sqrt{Q}^{-1}\!.
\]
We will now make use of the symmetric realisations for relaxation systems from \Cref{thm:1}. In particular this guarantees that
\[
\sqfunof{\begin{array}{c|c}
        \sqrt{Q}A\sqrt{Q}^{-1} & \sqrt{Q}B \\\hline
        C\sqrt{Q}^{-1} & D
    \end{array}}=\sqfunof{\begin{array}{c|c}
        \bar{A} & \bar{B} \\\hline
        \bar{B}^T & D
    \end{array}},
\]
where $\bar{A}\preceq{}0$. Substituting in this similarity transform shows that
\[
M\s=\funof{sI-\bar{A}+\bar{B}K\s\funof{I+DK\s}^{-1}\bar{B}^T}^{-1}.
\]
Next note that if $K\s\equiv{}\bar{K}\succeq{}0$ is stabilising, then
\[
\begin{aligned}
X&=-\bar{A}+\bar{B}\bar{K}\funof{I+D\bar{K}}^{-1}\bar{B}^T\\
&=-\bar{A}+\bar{B}\sqrt{\bar{K}}\funof{I+\sqrt{\bar{K}}D\sqrt{\bar{K}}}^{-1}\sqrt{\bar{K}}\bar{B}^T\succ{}0.
\end{aligned}
\]
Therefore for any such $K\s$, \cref{eq:11} can be rewritten as
\[
V\s=\begin{bmatrix}
I\\\alpha{}^{-1}\bar{K}
\end{bmatrix}\funof{I+D\bar{K}}^{-1}\bar{B}^TX^{-1}\funof{sX^{-1}+I}^{-1}.
\]
Since 
\[
\norm{\funof{sX^{-1}+I}^{-1}}_\infty{}=1,\,\text{and}\,\funof{sX^{-1}+I}^{-1}\Bigg|_{s=0}=I,
\]
we see that that
\[
\norm{V}_\infty=\norm{V\funof{0}}_2.
\]
Therefore given any stabilising $\bar{K}\succeq{}0$, the worst case disturbance is of the form in \cref{eq:disres}. Therefore the controller in 1) is not only optimal over all disturbances in $\mathcal{W}_Q^\infty{}$, but also over all in $\mathcal{W}_Q$, and therefore solves \Cref{prob:1}. 

We now consider \Cref{prob:2}. This is not an \Hfty{} control problem, so a little more work is required. First observe that
\[
G\s=\begin{bmatrix}
\bar{B}\sqrt{\bar{A}^{-1}}\\\sqrt{D}
\end{bmatrix}^T\begin{bmatrix}
\funof{s\bar{A}^{-1}-I}^{-1}&0\\0&I
\end{bmatrix}
\begin{bmatrix}
\bar{B}\sqrt{\bar{A}^{-1}}\\\sqrt{D}
\end{bmatrix}.
\]
This implies that for any $s$ in the closed right half plane
\[
\begin{aligned}
\abs{\hat{u}\s^*G\s{}\hat{u}\s}\leq{}\hat{u}\s^*G\funof{0}\hat{u}\s.
\end{aligned}
\]
The Plancharel theorem then implies that
\[
\begin{aligned}
\int_{0}^\infty{}u\tm^T\bar{y}\tm{}\,dt&=\int_{-\infty{}}^\infty{}\hat{u}\jw^*G\jw{}\hat{u}\jw{}\,d\omega,\\
&\leq{}\int_{-\infty{}}^\infty{}\abs{\hat{u}\jw^*G\jw{}\hat{u}\jw{}}\,d\omega,\\
&\leq{}\int_{-\infty{}}^\infty{}\hat{u}\jw^*G\funof{0}{}\hat{u}\jw{}\,d\omega,\\
&=
\int_{0}^\infty{}u\tm^TG\funof{0}u\tm{}\,dt.
\end{aligned}
\]
Therefore for any stabilising $K\s\equiv{}\bar{K}\succeq{}0$, the performance criterion in \Cref{prob:2} is always upper bounded by
\[
\norm{\begin{bmatrix}
I&0\\0&\sqrt{G\funof{0}}^{-1}
\end{bmatrix}V\s}_\infty{}=\norm{\begin{bmatrix}
I&0\\0&\sqrt{G\funof{0}}^{-1}
\end{bmatrix}V\funof{0}}_2.
\]
This means that the controller in 2) is not only optimal over all disturbances in $\mathcal{W}_Q^\infty{}$, but also over all in $\mathcal{W}_Q$, and therefore solves \Cref{prob:2}.
\end{proof}

\section{Examples}\label{sec:ex}

\subsection{Optimal Control of Inductive Electrical Networks}

In this example we will show how to synthesise an optimal controller for a simple inductive electrical network. In particular we will show how to interpret and synthesise the optimal controller for \Cref{prob:1} using duality theory for electrical networks. The particular topology considered here has been chosen for illustrative purposes, and far more complicated networks could be handled with an identical methodology.

\begin{figure}[th]
    \centering
    
    \begin{tikzpicture}[>=stealth]
    \draw[-, thick] (1.5,0) -- (3,0);
    \draw[-, thick] (1.5,2) -- (3,2);
    
    \draw[->, thick] (1.5,0.15) -- node[left] {$v$} (1.5,1.85);
    \draw[<-, thick] (0.75,0) -- node[above] {$i$} (1.35,0);
    \draw[->, thick] (0.75,2) -- node[above] {$i$} (1.35,2);
    
   \draw[-, thick] (3,1) -- node[left] {$R_3$} (3,2);
    \draw[-, thick] (3,0) -- node[left] {$L_2$} (3,1);
    \draw[-, thick] (3,2) -- node[above] {$R_1$} (4,1.5);
    \draw[-, thick] (3,0) -- node[below, xshift=1mm] {$L_1$} (4,0.5);
    \draw[-, thick] (4,0.5) -- node[right] {$R_2$} (4, 1.5);
    
    \draw[black, fill=white] (1.5,0) circle (0.05);
    \draw[black, fill=white] (1.5,2) circle (0.05);
    
    \draw[black, fill=black] (3,1) circle (0.05);
    \draw[black, fill=black] (3,0) circle (0.05);
    \draw[black, fill=black] (3,2) circle (0.05);
    \draw[black, fill=black] (4,1.5) circle (0.05);
    \draw[black, fill=black] (4,0.5) circle (0.05);

    \end{tikzpicture}
    
    \caption{An electrical network consisting of resistors and inductors and one port.}
    \label{fig:examplenetwork}
    
\end{figure}
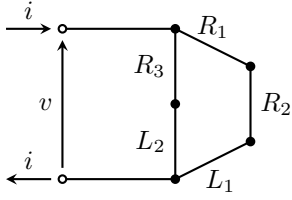

Consider the graph in~\Cref{fig:examplenetwork}. This represents an electrical network in which each edge represents either a resistor or an inductor. In addition a pair of terminals have been specified. It is through this port that currents can be injected into the network, and it is our aim to design a controller to regulate this current flow about an equilibrium. Although very abstract, such physical models are common throughout physics and engineering, and through the use of analogues can be used to represent a wide range of systems, for example commodity flow networks, or heating networks \cite{BXM07,SMR67}. 

Since the network only contains elements of one storage type, the dynamics of the electrical network are of the relaxation type. This can be shown explicitly by finding $G$. In this case it is simple to show by lumping elements that $G:i\rightarrow{}v$ is given by
\[
G\s=\frac{1}{L_1s+R_1+R_2}+\frac{1}{L_2s+R_3}.
\]
One possible realisation of this system is given by
\[
\Sigma_G=\sqfunof{\begin{array}{cc|c}
     -R_1/L_1-R_2/L_1&0&1  \\
     0&-R_3/L_2&1\\\hline
     1/L_1&1/L_2&0
\end{array}},
\]
which clearly satisfies the conditions of \Cref{thm:1} with
\[
Q=\begin{bmatrix}
1/L_1&0\\
0&1/L_2
\end{bmatrix}.
\]
Applying \Cref{thm:main} shows that the controller
\begin{equation}\label{eq:cont}
K\equiv{}\frac{\alpha^{-1}}{R_1+R_2}+\frac{\alpha^{-1}}{R_3}
\end{equation}
is optimal for \Cref{prob:1}. Let us now consider how to synthesise this controller. A simple way to do this is to build a resistor that satisfies
\[
V_c=\underbrace{\funof{\frac{\alpha^{-1}}{R_1+R_2}+\frac{\alpha^{-1}}{R_3}}}_{R_c}I_c.
\]
The control law in \cref{eq:cont} could therefore be implemented by connecting the above resistance to the original system across the terminals, as shown in~\Cref{fig:examplecontrol1}. 

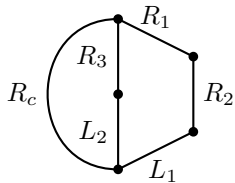
\begin{figure}[ht]
    \centering
    
    \begin{tikzpicture}[>=stealth]
    \draw[-, thick] (3,1) -- node[left] {$R_3$} (3,2);
    \draw[-, thick] (3,0) -- node[left] {$L_2$} (3,1);
    \draw[-, thick] (3,2) -- node[above] {$R_1$} (4,1.5);
    \draw[-, thick] (3,0) -- node[below, xshift=1mm] {$L_1$} (4,0.5);
    \draw[-, thick] (4,0.5) -- node[right] {$R_2$} (4, 1.5);
    
    \draw[thick] (3,2) to[out=180,in=180, distance=1.25cm] node[left] {$R_c$} (3,0);
    
    \draw[black, fill=black] (3,1) circle (0.05);
    \draw[black, fill=black] (3,0) circle (0.05);
    \draw[black, fill=black] (3,2) circle (0.05);
    \draw[black, fill=black] (4,1.5) circle (0.05);
    \draw[black, fill=black] (4,0.5) circle (0.05);
    
    \end{tikzpicture}
    
    \caption{Implemention of the optimal control law.}
    \label{fig:examplecontrol1}
    
\end{figure}

This is because Kirchhoff's laws for this operation are given by
\[
\begin{aligned}
i+I_c&=0\\v&=V_c.
\end{aligned}
\]
These imply that $v=-R_ci$, which is precisely the required control law (c.f. \cref{eq:p2}).

However let us now think further about what the equation for the controller in \Cref{thm:main} means. First note that in this case $G\funof{0}$ is equal to the admittance of the network we wish to control in steady state. That is there will only be a voltage drop across the resistive components. Our task is then to synthesise a resistor with impedance equal to the steady state admittance of the network. Such networks can be found by finding the so called dual network (see e.g. \cite[\S{}10.4.3]{DK69}). This process is illustrated in \Cref{fig:networkconstruction3}. Note that this gives an algorithmic way to synthesise the optimal controller that inherits the sparsity of the electrical network we wish to control. This is not so important for this specific example since the resulting network can always be lumped into a single resistor. The real strength of this approach is that it could be applied to synthesise the optimal controller in a sparse manner even when the graph is large (and planar), and the network has many ports.

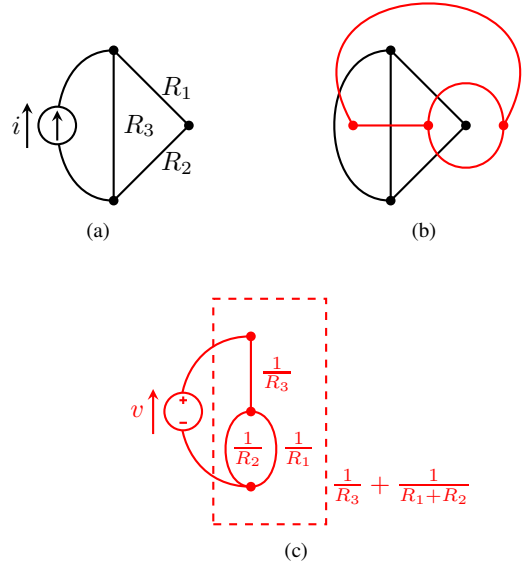
\begin{figure}[t]
    \centering
    
    \subfigure[]{\begin{tikzpicture}[>=stealth]
    \draw[thick] (0,0) -- node[right] {$R_3$} (0,2);
    \draw[thick] (0,0) -- node[right] {$R_2$} (1,1);
    \draw[thick] (0,2) -- node[right] {$R_1$} (1,1);
     \draw[thick] (0,0) to[out=180,in=180, distance=1cm] node[left,xshift=-10pt] {$i$} (0,2);
     \draw[fill=white] (-.75,1) circle (0.25);
    \draw[->] (-1.15,0.7) -- (-1.15, 1.3);
    \draw[->, thick] (-.75,0.8) -- (-.75, 1.2);
     \draw[black, fill=black] (0,0) circle (0.05);
     \draw[black, fill=black] (0,2) circle (0.05);
     \draw[black, fill=black] (1,1) circle (0.05);
    \end{tikzpicture}}
 \hspace{.5cm}   
    \subfigure[]{\begin{tikzpicture}[>=stealth]
    \draw[thick] (0,0) -- (0,2);
    \draw[thick] (0,0) -- (1,1);
    \draw[thick] (0,2) -- (1,1);
     \draw[thick] (0,0) to[out=180,in=180, distance=1cm] (0,2);
     
     \draw[red, thick] (-0.5,1) -- (0.5,1);
     \draw[red, thick] (0.5,1) to[out=90,in=90, distance=0.75cm] (1.5,1);
     \draw[red, thick] (0.5,1) to[out=-90,in=-90, distance=0.75cm] (1.5,1);
     \draw[red, thick] (1.5,1) to[out=60,in=120, distance=2.5cm] (-0.5,1);
     
     \draw[red, fill=red] (-0.5,1) circle (0.05);
     \draw[red, fill=red] (0.5,1) circle (0.05);
     \draw[red, fill=red] (1.5,1) circle (0.05);
     \draw[black, fill=black] (0,0) circle (0.05);
     \draw[black, fill=black] (0,2) circle (0.05);
     \draw[black, fill=black] (1,1) circle (0.05);
    \end{tikzpicture}}
    
    \subfigure[]{\begin{tikzpicture}[>=stealth]
    
    \node (dummy) at (0,2.8) {\textcolor{white}{~}};
    
    \draw[thick, red, -] (0,1) -- node[right] {$\frac{1}{R_3}$} (0,2);
    \draw[thick, red, -] (0,1) to[out=180,in=180, distance=0.45cm] node[right, xshift=-0.5mm] {$\frac{1}{R_2}$} (0,0);
    \draw[thick, red, -] (0,1) to[out=0,in=0, distance=0.45cm] node[right, xshift=-0.5mm] {$\frac{1}{R_1}$} (0,0);
     \draw[thick, red, -] (0,2) to[out=180,in=180, distance=1.25cm] node[left, xshift=-3.5mm] {$v$} (0,0);
     \draw[red,->] (-1.3,.7) -- (-1.3,1.3);
     \draw[red,fill=white] (-.9,1) circle (0.25);
     \draw[red] (-.9,1.1) -- (-.9, 1.2);
     \draw[red] (-.95,1.15) -- (-.85, 1.15);
     \draw[red] (-.95,.85) -- (-.85, .85);
     \draw[red, fill=red] (0,1) circle (0.05);
     \draw[red, fill=red] (0,2) circle (0.05);
     \draw[red, fill=red] (0,0) circle (0.05);
     \draw[dashed, red] (-0.5,-0.5) rectangle (1,2.5);
     \node (shift) at (2,0) {\textcolor{red}{$\frac{1}{R_3}+\frac{1}{{R_1+R_2}}$}};
    
    \end{tikzpicture}}

    \caption{Construction of the dual network used to implement the optimal controller. First, as shown in (a), all the inductive edges are contracted to give a purely resistive network with the same steady state admittance as $G$. Next the dual graph of this network is constructed as shown in (b). Finally, as shown in (c), each edge in the dual graph is assigned a resistance equal to the reciprocal of that from (a). This process produces a circuit with impedance equal to the admittance of the circuit in (a), giving an electrical realisation of the optimal controller.}
    \label{fig:networkconstruction3}

\end{figure}

\subsection{Solving Least Squares Problems Using Circuits}

The use of electrical circuits to solve optimisation problems is classical \cite{BK06}. In this section we will use \Cref{thm:main} to show how to optimise the dynamic performance of a simple circuit that will solve the least squares problem
\begin{equation}\label{eq:lsex}
\min_{x\in\R^m}\norm{Ax+b}_2.
\end{equation}
Recall that this problem has (minimum norm) solution $x\equiv{}-A^{\dagger}b$.

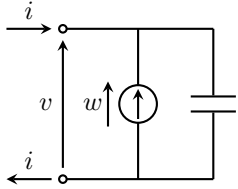
\begin{figure}
    \centering
    \begin{tikzpicture}[>=stealth]
    \draw[-, thick] (1,0) -- (3,0);
    \draw[-, thick] (1,2) -- (3,2);
    
    \draw[-, thick] (2,0) -- (2,2);
    \draw[fill=white] (2,1) circle (0.25);
    \node (R) at (1.4,1) {$w$};
    \draw[->] (1.6,0.7) -- (1.6, 1.3);
    \draw[->, thick] (2,0.8) -- (2, 1.2);
    
    \draw[-, thick] (3,0) -- (3,2);
    \draw[white,fill=white] (2.8,0.9) rectangle (3.2,1.1);
    \draw[-, thick] (2.7,0.9) -- (3.3,0.9);
    \draw[-, thick] (2.7,1.1) -- (3.3,1.1);
    \node (C) at (3.75,1) {};
    
    \draw[->, thick] (1,0.15) -- node[left] {$v$} (1,1.85);
    \draw[<-, thick] (0.25,0) -- node[above] {$i$} (0.85,0);
    \draw[->, thick] (0.25,2) -- node[above] {$i$} (0.85,2);
    
    \draw[black, fill=white] (1,0) circle (0.05);
    \draw[black, fill=white] (1,2) circle (0.05);
    
    \end{tikzpicture}
    \caption{Circuit that synthesises \cref{eq:cap}.}
    \label{fig:capcap}
\end{figure}

Consider now the electrical components with dynamics
\begin{equation}\label{eq:trans}
\begin{bmatrix}
V_1\\I_2
\end{bmatrix}=\begin{bmatrix}
0&A\\-A^T&0
\end{bmatrix}
\begin{bmatrix}
I_1\\V_2
\end{bmatrix}
\end{equation}
and
\begin{equation}\label{eq:cap}
\begin{aligned}
\dot{q}&=I_3+w\\
V_3&=q.
\end{aligned}
\end{equation}
The component described by \cref{eq:trans} can be synthesised using transformers for any $A\in\R^{n\times{}m}$ (see e.g. \cite[\S{}VI.2)]{Wil76}), and that in \cref{eq:cap} using capacitors and current sources (see \Cref{fig:capcap}). Interconnecting these components according to Kirchhoff's relations
\[
I_2+I_3=0\;\text{and}\;V_2=V_3
\]
yields a system $G:\funof{w,I_1}\rightarrow{}V_1$ with realisation
\[
\begin{aligned}
\dot{q}&=AI_1+w,\;q\funof{0}=0,\\
V_1&=A^Tq.
\end{aligned}
\]
\Cref{thm:main} clearly applies with $Q\equiv{}I$. This shows that the control law
\[
I_1=-\alpha{}^{-1}V_1
\]
is optimal with respect to \Cref{prob:2}. This controller can be synthesised by connecting the resistors
\[
V_4=\alpha{}I_4
\]
to the existing circuit according to
\[
I_1+I_4=0\;\text{and}\;V_1=V_4.
\]
Now consider the behaviour of the circuit if we apply the current
\[
w\tm=bH\tm
\]
using the current sources, where $H\tm$ denotes the unit step. The final value theorem shows that
\[
\begin{aligned}
\lim_{t\rightarrow{}\infty}V_1\tm&=\lim_{s\rightarrow{}0}sA^T\funof{sI+\alpha^{-1}AA^T}^{-1}\frac{1}{s}b\\
&=\alpha{}A^{\dagger{}}b.
\end{aligned}
\]
This implies that as $t\rightarrow{}\infty$, the current flowing through the resistors equals the solution to \cref{eq:lsex}. That is, this simple electrical circuit can be used to solve the least squares problem in \cref{eq:lsex} for any $b$, while minimising the dynamic performance objective in \Cref{prob:2}. In particular adjusting the value of the resistance $\alpha$ allows the speed with which the problem is solved to be balanced against the energy losses (which will heat the system up).

\section{Conclusions}

It has been shown that if a system is of the relaxation type, then some simple \Hfty{}-type control problems can be solved analytically. The resulting controllers inherit the structural properties of the system. Therefore if the original system has a sparse structure, the optimal controllers can be synthesised with sparse resistive networks. This has been demonstrated by designing simple, sparse, electrical circuits to optimally control large scale inductive networks and to solve linear regression problems.

\vspace{-2mm}
\bibliographystyle{IEEEtran}
\bibliography{references.bib}
\vspace{-2mm}

\appendix{}

\begin{lemma}\label{lem:lb}
Let $G_1\in\C^{n\times{}n}$, $G_2\in\C^{n\times{}n}$ and $G_3\in\C^{n\times{}m}$. If $G_1$ is invertible, then given any vector $v\in\C^m$
\begin{equation}\label{op:1}
\begin{aligned}
-G_1G_1^*G_2^*=&\arg\min_{K\in\C^{n\times{}n}}\norm{z}_2\\
\text{s.t.}\;\;&z=\begin{bmatrix}
I\\-G_1^{-1}K
\end{bmatrix}\funof{I-G_2K}^{-1}G_3v.
\end{aligned}
\end{equation}
\end{lemma}
\begin{proof}
Consider
\begin{equation}\label{op:2}
\begin{aligned}
\min_{K\in\C^{n\times{}n}}&\norm{z}_2\\
\text{s.t.}\;\;\begin{bmatrix}
I&G_2G_1
\end{bmatrix}z&=G_3v\\
z&=\begin{bmatrix}
I\\-G_1^{-1}K
\end{bmatrix}x.
\end{aligned}
\end{equation}
Eliminating $x$ from the above shows that the constraints in \cref{op:1,op:2} are the same, and therefore that these problems are equivalent. We may obtain a lower bound to the problem in \cref{op:2} by dropping the final constraint. If this is done, \cref{op:2} becomes a standard minimum norm least squares problem, with optimal solution
\[
z\equiv{}\begin{bmatrix}
I&G_2G_1
\end{bmatrix}^\dagger{}G_3v=\begin{bmatrix}
I\\G_1^*G_2^*
\end{bmatrix}\funof{I+G_2G_1G_1^*G_2}^{-1}G_3v.
\]
Setting $K\equiv{}-G_1G_1^*G_2^*$ in the constraint in \cref{op:1} achieves precisely this $z$, which completes the proof.
\end{proof}

\end{document}